\documentclass[11pt]{article}

\usepackage{graphicx}
\usepackage{amsmath}
\usepackage{amsfonts}
\usepackage{titlesec}
\usepackage{amsthm}
\usepackage{amssymb}
\usepackage{color}
\usepackage{mathtools}
\usepackage{geometry}
\usepackage{setspace}
\usepackage{times}
\usepackage{algorithm2e}

\usepackage{verbatim}

\numberwithin{equation}{section}

\newtheorem{theorem}{Theorem}[section]

\theoremstyle{definition}

\newtheorem{remark}[theorem]{Remark}

\definecolor{light-gray}{gray}{0.69}
\definecolor{light-red}{rgb}{1.0,0.4,0.4}

\definecolor{light-blue}{rgb}{0.4,0.45,1}
\definecolor{light-green}{rgb}{0.5,0.8,0.0}
\definecolor{dark-green}{rgb}{0.0,0.4,0.0}
\definecolor{dark-red}{rgb}{1.0,0.3,0.3}

\definecolor{dark-gray}{gray}{0.59}
\definecolor{very-dark-gray}{gray}{0.39}
\definecolor{lighter-red}{rgb}{1.0,0.6,0.6}

\definecolor{ocker_hell}{rgb}{0.75,0.7,0.4}
\definecolor{gelb_dunkel}{rgb}{0.75,0.7,0.0}
\definecolor{gruen_hell}{rgb}{0.5,0.8,0.0}

\definecolor{dark-blue}{rgb}{0.0,0.0,0.5}
\definecolor{new-blue}{rgb}{0.0,0.0,0.8}

\definecolor{lila}{rgb}{0.5,0.0,0.5}
\definecolor{dark-red}{rgb}{0.5,0.0,0.0}

\begin{document}

\renewcommand{\thefootnote}{\fnsymbol{footnote}}\setcounter{footnote}{0}

\begin{center}
{\Large Construction of Locally Conservative Fluxes for High Order Continuous Galerkin Finite Element Methods }
\end{center}

\renewcommand{\thefootnote}{\fnsymbol{footnote}}
\renewcommand{\thefootnote}{\arabic{footnote}}

\begin{center}
Q.~Deng  and V.~Ginting\\
Department of Mathematics, University of Wyoming, Laramie, Wyoming 82071, USA\footnote{Complete mailing address: Department of Mathematics, 1000 E. University Ave. Dept. 3036, Laramie, WY 82071}
\end{center}


\renewcommand{\baselinestretch}{1.5}
\begin{abstract}


We propose a simple post-processing technique for linear and high order  continuous Galerkin Finite Element Methods (CGFEMs) to obtain locally conservative flux field. The post-processing technique requires solving an auxiliary problem on each element independently which results in solving a linear algebra system whose size is $\frac{1}{2}(k+1)(k+2)$ for $k^\text{th}$ order CGFEM. The post-processing could have been done directly from the finite element solution that results in locally conservative flux on the element. However, the normal flux is not continuous at the element's boundary. To construct locally conservative flux field whose normal component is also continuous, we propose to do the post-processing on the nodal-centered control volumes which are constructed from the original finite element mesh.  We show that the post-processed solution converges in an optimal fashion to the true solution in an $H^1$ semi-norm. We present various numerical examples to demonstrate the performance of the post-processing technique.

\end{abstract}



\section{Introduction} \label{sec:intro}

\renewcommand{\baselinestretch}{1.5}

Both finite volume element method (FVEM) and continuous Galerkin finite element method (CGFEM) are widely used for solving partial differential equations and they have both advantages and disadvantages. Both methods share
a property in how the approximate solutions are represented through linear combinations of the finite element basis
functions. They have a main advantage of the ability to solve the partial differential equations posed in
complicated geometries. However the methods differ in the variational formulations governing the approximate
solutions. CGFEMs are defined as a global variational formulation while FVEM relies on local variational formulation, 
namely one that imposes local conservation of fluxes. In the case of linear finite element, it is well-known that the bilinear form of FVEM is closely related to its CGFEM counterpart, and this closeness is exploited to carry out the
error analysis of FVEM \cite{bank1987some, hackbusch1989first}. In 2002, Ewing \textit{et al.}  showed that the stiffness matrix derived from the linear FVEM is a small perturbation of that of the linear CGFEM for sufficiently small mesh size of the triangulation \cite{ewing2002accuracy}. In 2009, an identity between stiffness matrix of linear FVEM and the matrix of linear CGFEM was established; see Xu \textit{et al.} \cite{xu2009analysis}. A significant amount of work has been done to investigate the closeness of linear FVEM and linear CGFEM. However, the current understanding and implementation
of higher order FVEMs are still at its infancy and are not as satisfactory as linear FVEM. For one-dimensional elliptic equations, high order FVEMs have been developed in \cite{plexousakis2004construction}. Other relevant high order FVEM work can be found in \cite{liebau1996finite, li2000generalized, cai2003development, chen2012higher, chen2015construction}.

As mentioned, FVEM produces locally conservative fluxes while, due to the global formulation, CGFEMs do not. Robustness of the CGFEMs for any order has been established through extensive and rigorous error analysis, while this is not the case for FVEM. Development of linear algebra solvers for CGFEMs has reached an advanced stage, mainly driven from a solid understanding of the
variational formulations and their properties, such as coercivity (and symmetry) of
the bilinear form in the Galerkin formulation. On the other hand, the resulting linear algebra systems derived from FVEMs, especially high order FVEMs, are not that easy to solve. Typically, the matrices resulting from FVEMs are not symmetric even if the original boundary value problem is. Furthermore, at most FVEM discretization with linear finite element basis
yields M-matrix, while with quadratic finite element basis it is not (see \cite{liebau1996finite}).

Preservation of numerical local conservation property of approximate solutions are imperative in simulations of many 
physical problems especially those that are derived from law of conservation. In order to maintain the advantages of CGFEM as well as to obtain locally conservative fluxes, post-processing techniques are developed; see \cite{arbogast1995characteristics,  bush2015locally, bush2013application, chou2000conservative,  cockburn2007locally, cordes1992continuous, 
 deng2015construction, gmeiner2014local, hughes2000continuous, 
  kees2008locally,  larson2004conservative, loula1995higher,
nithiarasu2004simple, srivastava1992three,  sun2009locally,  thomas2008element, thomas2008locally, toledo1989mixed,   zhang2013locally}. The post-processing techniques proposed in the aforementioned references are mainly techniques for post-processing linear finite element related methods and they include finite element methods for solving pure elliptic equations, advection diffusion equations, advection dominated diffusion equations, elasticity problems,
Stokes problem, etc. Among them, some of the proposed post-processing techniques require solving global systems.  We will focus on a brief review on the post-processing techniques for high order CGFEMs. Generally, those post-processing techniques that works for high order CGFEMs also work for lower order CGFEMs, but may not vice versa.

There are very limited work on the post-processing for high order CGFEMs to obtain locally conservative fluxes.
An interesting work on post-processing for high order CGFEMs is in Zhang \textit{et al} \cite{zhang2012flux}. In their work, they showed that the elemental fluxes directly calculated from any order of CGFEM solutions converge to the true fluxes in an optimal order but the fluxes are not naturally locally conservative. They proposed two post-processing techniques to obtain the locally conservative fluxes at the boundaries of the each element. The post-processed solutions are of optimal both $L^2$ norm and $H^1$ semi-norm convergence orders. Very interestingly, one of the post-processed solution still satisfies the original finite element equations. Other work on post-processing to gather locally conservative
flux that include high order finite elements is recorded in \cite{cockburn2007locally}. The post-processing involves
two steps: solving a set of local systems followed by solving a global system.

A uniform approach to local reconstruction of the local fluxes from various finite element method (FEM) solutions was presented in \cite{becker2015robust}. These methods includes any order of conforming, nonconforming, and discontinuous FEMs. They proposed a hybrid formulation by  utilizing Lagrange multipliers, which can be computed locally. However, the reconstructed fluxes are not locally conservative. They used the reconstructed fluxes to derive a posteriori error estimator \cite{becker2015stopping}.

In this paper, we propose a post-processing technique for any order of CGFEMs to obtain fluxes that are locally conservative on a dual mesh consisting of control volumes.   The dual mesh is constructed from the original mesh in a different way for different order of CGFEMs. The technique requires solving an auxiliary problem which results in a low dimensional linear algebra system on each element independently. Thus, the technique can be implemented in a parallel environment and it produces locally conservative fluxes wherever it is needed. The technique is developed on triangular meshes and it can be naturally extended to rectangular meshes.

The rest of the paper is organized as follows. The CGFEM formulation of the model problem is presented in Section \ref{sec:cgfem} followed by the description of the methodology of the post-processing technique in Section \ref{sec:pp}. Analysis of the post-processing technique is presented in Section \ref{sec:ana} and numerical examples are presented to demonstrate the performance of the technique in Section \ref{sec:num}.

\section{Continuous Galerkin Finite Element Method} \label{sec:cgfem}
For simplicity, we consider the elliptic boundary value problem
\begin{equation}\label{pde}
\begin{cases}
\begin{aligned}
- \nabla \cdot ( \kappa \nabla u ) & =f \quad \text{in} \quad \Omega, \\
u&= g \quad \text{on} \quad \partial\Omega, \\
\end{aligned}
\end{cases}
\end{equation}
where $\Omega$ is a bounded open domain in $\mathbb{R}^2$ with Lipschitz boundary $\partial\Omega$,  $\kappa = \kappa(\boldsymbol{x})$ is the elliptic coefficient, $u=u(\boldsymbol{x})$ is the solution to be found, $f=f(\boldsymbol{x})$ is a forcing function. Assuming $0 < \kappa_{\min} \leq \kappa(\boldsymbol{x}) \leq \kappa_{\max} < \infty $ for all $\boldsymbol{x} \in\Omega$ and $f \in L^2(\Omega)$,  Lax-Milgram Theorem guarantees a unique weak solution to \eqref{pde}.  For the polygonal domain $\Omega$, we consider a partition $\mathcal{T}_h$ consisting of triangular elements $\tau$ such that $\overline\Omega = \bigcup_{\tau\in\mathcal{T}_h} \tau.$ We set $h=\max_{\tau\in\mathcal{T}_h} h_\tau$ where $h_\tau$ is defined as the diameter of $\tau$. The continuous Galerkin finite element space is defined as
$$
V^k_h = \big\{w_h\in C(\overline\Omega): w_h|_\tau \in P^k(\tau), \ \forall \ \tau\in\mathcal{T}_h \ \text{and} \ w_h|_{\partial\Omega} = 0 \big\}, 
$$ 
where $P^k(\tau)$ is a space of polynomials with degree at most $k$ on $\tau$.
The CGFEM formulation for \eqref{pde} is to find $u_h$ with $(u_h - g_h) \in V^k_h$, such that
\begin{equation} \label{eq:fem} 
a(u_h, w_h) = \ell(w_h) \quad \forall \ w_h \in V^k_h,  
\end{equation} 
where
\begin{equation*}
a(v, w) = \int_\Omega \kappa  \nabla v \cdot \nabla w \ \text{d} \boldsymbol{x}, \qquad \text{and} \qquad \ell(w) = \int_\Omega f  w \ \text{d} \boldsymbol{x}, 
\end{equation*}
and $g_h \in V_h^k$ can be thought of as the interpolant of $g$ using the usual finite element basis. 

\begin{figure}[ht]
\centering
\includegraphics[height=4.3cm]{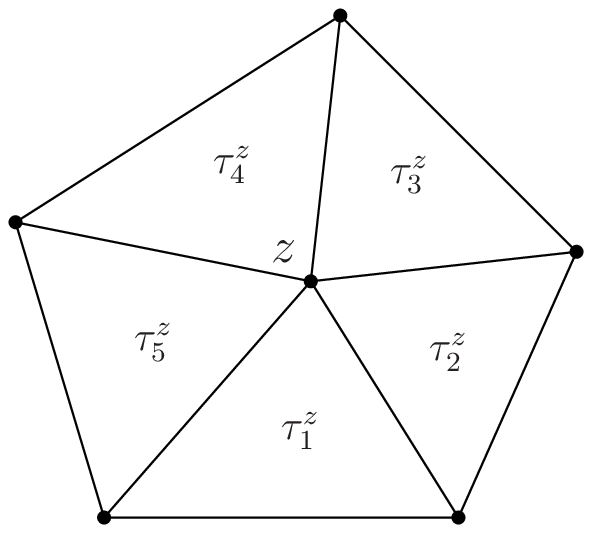} 
\caption{Setting for $V_h^1$: the dots represent the degrees of freedom and $\Omega^z = \cup_{i=1}^5 \tau_i^z$ is the support of $\phi_z$.} 
\label{fig:lelemsupp}
\includegraphics[height=4.3cm]{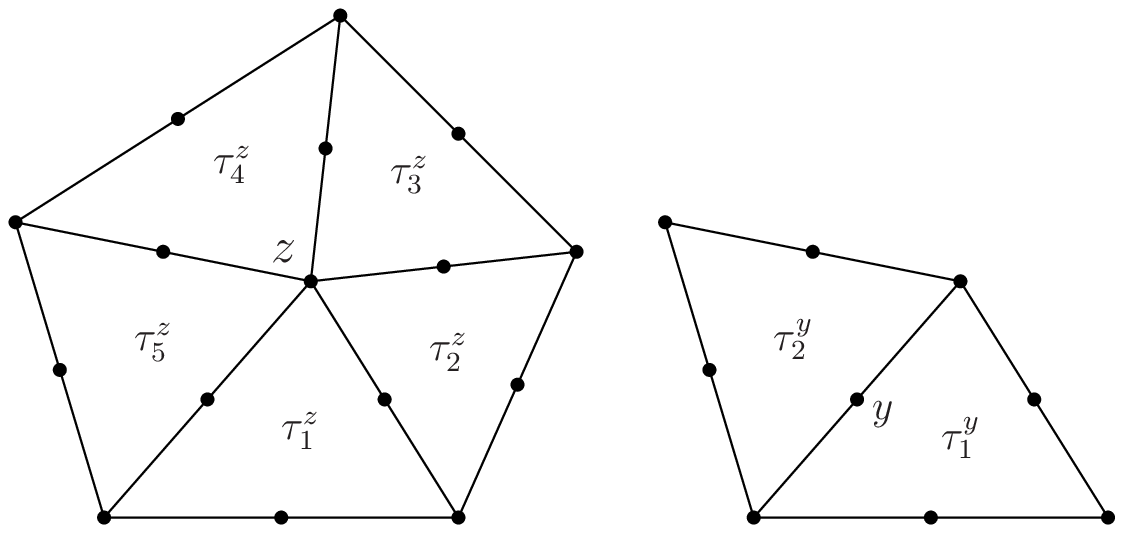}
\caption{Setting for $V_h^2$: The dots represent the degrees of freedom. $\Omega^z=  \cup_{i=1}^5 \tau_i^z$ (left) is support of $\phi_z$ and $\Omega^y =  \cup_{i=1}^2 \tau_i^y$ (right) is the support of $\phi_y$.}
\label{fig:qelemsupp}
\includegraphics[height=4.3cm]{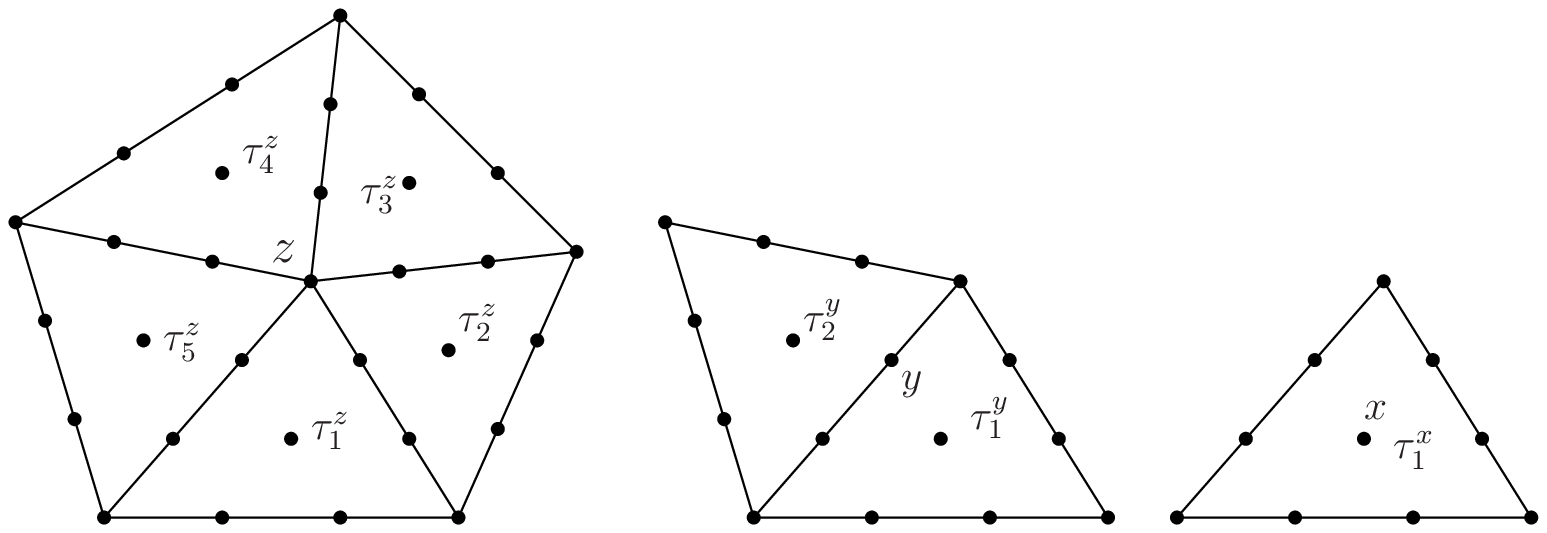}
\caption{Setting for $V_h^3$: The dots represent the degrees of freedom. $\Omega^z=  \cup_{i=1}^5 \tau_i^z$ (left) is the support of $\phi_z$, and $\Omega^y =  \cup_{i=1}^2 \tau_i^y$ (middle) is the support of $\phi_y$, and $\Omega^x = \tau_1^x$ is the support of $\phi_x$.}
\label{fig:celemsupp}
\end{figure}

We now present a fundamental but obvious fact  for this CGFEM formulation. To make it as general as possible, let $Z$ be the set of nodes in $\Omega$ resulting from the partition $\mathcal{T}_h$ and placing the degrees of freedom
owned by $V_h^k$. In particular, $Z$ consists of vertices for $V_h^1$ (see Figure \ref{fig:lelemsupp}), vertices and degrees of freedom on the edges for $V_h^2$ (see Figure \ref{fig:qelemsupp}), vertices and degrees of freedom on the edges and in the elements' barycenters for
$V_h^3$  (see Figure \ref{fig:celemsupp}). Furthermore, $Z = Z_{\text{in}} \cup Z_\text{d}$, where $Z_{\text{in}}$ is the set of interior degrees of freedom and $Z_\text{d}$ is the set of corresponding points on $\partial\Omega$. Denoting the usual Lagrange nodal basis of $V^k_h$ as $\{ \phi_\xi \}_{\xi \in Z_\text{in}}$, \eqref{eq:fem} yields 
\begin{equation} \label{eq:femphi}
a(u_h, \phi_\xi) = \ell (\phi_\xi) \quad \forall \ \xi \in Z_{\text{in}}. 
\end{equation}

For a $\xi \in Z_{\text{in}}$, let $\Omega^\xi$ be the support of the basis function $\phi_\xi$. Then in the partition $\mathcal{T}_h$, $\Omega^\xi = \cup_{i=1}^{N_\xi} \tau_i^\xi$, where
$\tau_i^\xi$ is an element that has $\xi$ as one of its degrees of freedom, and
 $N_\xi$ is the total number of such elements. In the linear case, $N_\xi$ is the number of elements sharing vertex $\xi$ (see Figure \ref{fig:lelemsupp}); in the quadratic case, $N_\xi$ is the number of elements sharing vertex $\xi$ or a middle point on an edge in which case $N_\xi = 2$ (see Figure \ref{fig:qelemsupp}); in the cubic case, $N_\xi$ can be those in quadratic case plus that it can be one if $\xi$ is inside the element (see Figure \ref{fig:celemsupp}). With this in mind,  \eqref{eq:femphi} is expressed as

\begin{equation} \label{eq:localfem}
\sum_{i=1}^{N_\xi} a_{\tau^\xi_i}(u_h, \phi_\xi) = \sum_{i=1}^{N_\xi} \ell_{\tau^\xi_i} (\phi_\xi),
\end{equation}
where $a_\tau(v_h, w_h)$ is $a(v_h, w_h)$ restricted to element $\tau$ and $\ell_\tau (w_h)$ is $\ell(w_h)$ restricted to element $\tau$.
Equation \eqref{eq:localfem} is fundamental and we will use this fact to derive the post-processing technique in Section \ref{sec:pp}.

\section{A Post-processing Technique}  \label{sec:pp}

A naive derivative calculation of $u_h$ does not yield locally conservative fluxes. For this reason, in this section we propose a post-processing technique to construct locally conservative fluxes over control volumes from CGFEM solutions. We will focus on the construction for $k^{\text{th}}$ order CGFEM, where $k=1, 2, 3$, i.e., linear, quadratic, and cubic CGFEMs. Construction for orders higher than these CGFEMs can be conducted in a similar fashion.

\subsection{Auxiliary Elemental Problem} \label{sec:bvp}

Based upon the original finite element mesh and $V_h^k$, a dual mesh that consists of control volumes is generated over which the post-processed fluxes is to satisfy the local conservation.  For $V_h^1$, we connect the barycenter and middle points of edges of a triangular element; see Figure \ref{fig:lelemcv}.  For $V_h^2$, we firstly discretize the triangular element into \textcolor{red}{four} sub-triangles and then connect the barycenters and middle points of each sub-triangle; see plots in Figure \ref{fig:qelemcv}, and similarly $V_h^3$, see plots in Figure \ref{fig:celemcv}. We can also see the construction of the dual mesh on a single element in Figure \ref{fig:elem}. Each control volume corresponds to a degree of freedom in CGFEMs. We post-process the CGFEM solution $u_h$ to obtain 
$\widetilde {\boldsymbol{\nu}}_h = -\kappa   \nabla \widetilde{u}_h $ such that it is continuous at the boundaries of each control volume and satisfies the local conservation property in the sense
\begin{equation} \label{eq:cvconservation}
\int_{\partial C^\xi} \widetilde {\boldsymbol{\nu}}_h \cdot \boldsymbol{n}  \ \text{d} l = \int_{C^\xi} f \ \text{d} \boldsymbol{x},
\end{equation}
where $C^\xi$ can be a control volume surrounding a vertex as $C^z$ in Figure \ref{fig:lelemcv}, \ref{fig:qelemcv}, and \ref{fig:celemcv}, or a control volume surrounding a degree of freedom on an edge as $C^y$ in Figure \ref{fig:qelemcv} and \ref{fig:celemcv}, or $C^x$ in Figure \ref{fig:celemcv}.

\begin{figure}[ht]
\centering
\includegraphics[height=4.5cm]{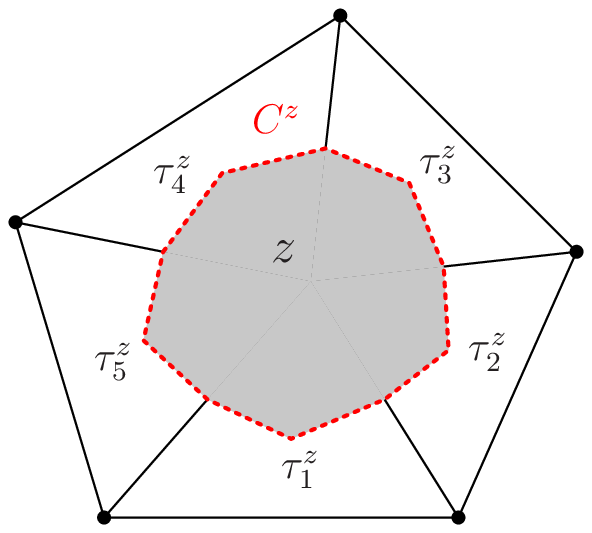}  
\caption{ $C^z$ is the control volume corresponding to $\phi_z$ in $V_h^1$. }
\label{fig:lelemcv}
\includegraphics[height=4.5cm]{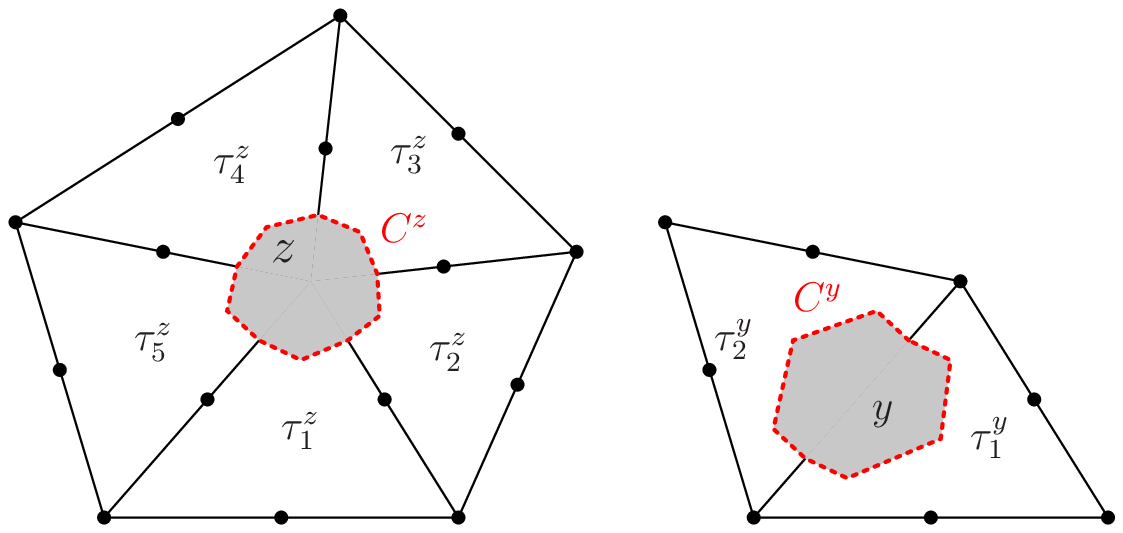}  
\caption{ $C^z, C^y$ are control volumes corresponding to $\phi_z$ (left) and $\phi_y$ (right), respectively, in $V_h^2$. }
\label{fig:qelemcv}
\includegraphics[height=4.5cm]{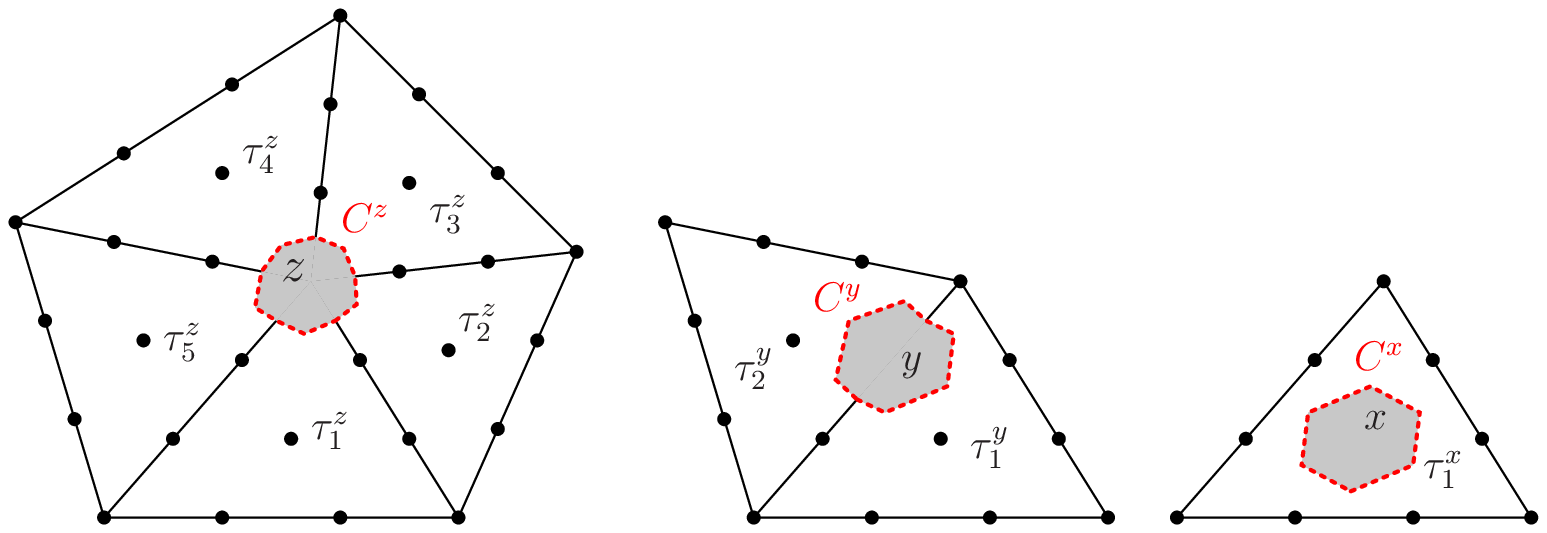}
\caption{ $C^z, C^y, C^x$ are control volumes corresponding to $\phi_z$ (left), $\phi_y$ (middle), and
$\phi_x$ (right), respectively,  in $V_h^3$. }
\label{fig:celemcv}
\end{figure}

In order to obtain the locally conservative fluxes on each control volume, we set and solve an elemental/local problem on $\tau$. Let $N_k = \frac{1}{2}(k+1)(k+2)$ be the total number of degrees of freedom on a triangular element for $V_h^k$. We denote the collection of those degrees of freedom by $s(\tau, k) = \{ z_j \}_{j=1}^{N_k}$; see Figure \ref{fig:elem}.  We partition each element $\tau$ into $N_k$ non-overlapping polygonals $\{ t_{z_j} \}_{j=1}^{N_k}$; see Figure \ref{fig:elem}.
For $t_\xi$ with $\xi \in s(\tau,k)$, we make decomposition $\partial t_\xi = ( \partial \tau \cap \partial t_\xi ) \cup ( \partial C^\xi \cap \partial t_\xi ).$  We also define the average on an edge or part of the edge which is the intersection of two elements $\tau_1$ and $\tau_2$ for vector $\boldsymbol{v}$ as
\begin{equation} \label{def:ave}
\{ \boldsymbol{v} \} = \frac{\boldsymbol{v}|_{\tau_1} + \boldsymbol{v}|_{\tau_2} }{2}. 
\end{equation}

\begin{figure}[ht]
\centering
\includegraphics[height=4.0cm]{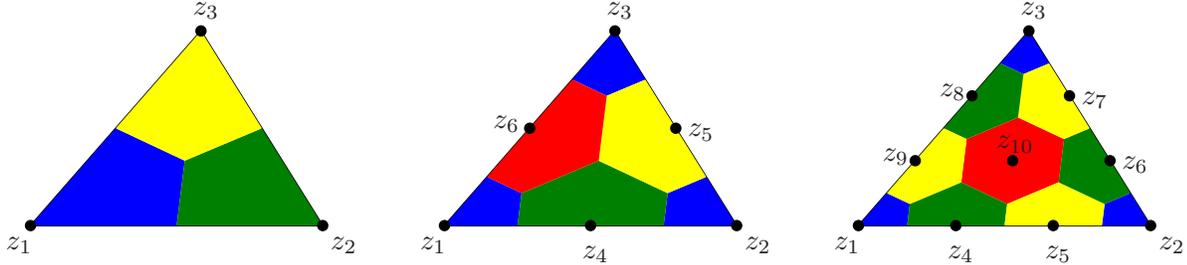}
\caption{ Control volume construction and degrees of freedom on an element for $V_h^k$: $k=1$ (left), $k=2$ (middle), and $k=3$ (right). }
\label{fig:elem}
\end{figure}

Let $V^0(\tau)$ be the space of piecewise constant functions on element $\tau$ such that $V^0(\tau) = \text{span} \{ \psi_\eta \}_{\eta \in s(\tau,k)}$, where $\psi_\eta$ is the characteristic function of the polygonal $t_\eta$, i.e.,
\begin{equation}
\psi_\eta(\boldsymbol{x}) = \Big\lbrace\begin{array}{c} 1 \quad \text{if} \quad \boldsymbol{x} \in t_\eta  \\ 0 \quad \text{if} \quad \boldsymbol{x} \notin t_\eta \end{array}.
\end{equation}
We define a map $I_\tau: H^1(\tau) \rightarrow V^0(\tau)$ with $I_\tau w = \displaystyle \sum_{\xi \in s(\tau,k)} w_\xi \psi_\xi$, where $w_\xi = w(\xi)$ for $w\in H^1(\tau)$. We define the following bilinear forms
\begin{equation} \label{eq:bebiforms}
b_\tau( v, w)  = -\sum_{\xi \in s(\tau,k)} \int_{\partial C^\xi \cap \partial t_\xi} \kappa  \nabla v  \cdot \boldsymbol{n} I_\tau w \ \text{d} l, \qquad e_\tau(v, w) = \int_{\partial \tau} \{ \kappa  \nabla v \} \cdot \boldsymbol{n}  w  \ \text{d} l.
\end{equation}

Let $ V^k_h(\tau) = \text{span}\{  \phi_\eta \}_{ \eta \in s(\tau,k) }$ where $ \phi_\eta$ can be thought as the usual
nodal $\eta$ basis function restricted to element $\tau$.
The elemental calculation for the post-processing is to find $\widetilde u_{\tau, h} \in V^k_h(\tau) $ satisfying 
\begin{equation} \label{eq:bvpvf}
b_\tau(\widetilde u_{\tau, h}, w) = \ell_\tau ( I_\tau w - w ) + a_\tau(u_h, w) + e_\tau(u_h, I_\tau w - w), \quad \forall \ w \in V^k_h(\tau).
\end{equation}

\newtheorem{lem}{Lemma}[section]
\begin{lem}  \label{lem:ebvp}
The variational formulation \eqref{eq:bvpvf} has a unique solution up to a constant.
\end{lem}

\begin{proof}
We have $V_h^k(\tau) = \text{span}\{  \phi_\xi \}_{ \xi \in s(\tau,k) }$, where $\phi_\xi(\eta) = \delta_{\xi \eta}$ with $\delta_{\xi \eta}$ being the Kronecker delta,
for all $\xi, \eta \in s(\tau,k)$. By replacing the test function $w$ with  $\phi_\xi$ for all $\xi \in  s(\tau,k)$,
 \eqref{eq:bvpvf}  is reduced to 
\begin{equation}  \label{eq:cz10}
- \int_{\partial C^\xi \cap \partial t_\xi } \kappa  \nabla \widetilde u_{\tau, h} \cdot \boldsymbol{n}  \ \text{d} l = \int_{t_\xi} f \ \text{d} \boldsymbol{x} - \ell_\tau (  \phi_\xi ) + a_\tau(u_h,  \phi_\xi ) + e_\tau( u_h, I_\tau  \phi_\xi -  \phi_\xi ), \quad \forall \ \xi \in  s(\tau,k).
\end{equation} 
This is a fully Neumann boundary value problem in $\tau$ with boundary condition satisfying 
\begin{equation} \label{eq:bvpb} 
- \int_{\partial \tau \cap \partial t_\xi } \kappa  \nabla \widetilde u_{\tau, h}  \cdot \boldsymbol{n}  \ \text{d} l = \ell_\tau (  \phi_\xi ) - a_\tau(u_h,  \phi_\xi ) - e_\tau( u_h, I_\tau  \phi_\xi -  \phi_\xi ), \quad \forall \ \xi \in s(\tau,k).
\end{equation} 
To establish the existence and uniqueness of the solution, one needs to verify the compatibility condition \cite{evans2010partial}. We calculate 
\begin{equation*}
-\int_{\partial \tau} \kappa  \nabla \widetilde u_{\tau, h} \cdot \boldsymbol{n} \ \text{d} l = \sum_{\xi \in s(\tau,k)} \Big( \ell_\tau (  \phi_\xi ) - a_\tau(u_h,  \phi_\xi ) - e_\tau( u_h, I_\tau  \phi_\xi -  \phi_\xi ) \Big).
\end{equation*}
Using the fact that $\sum_{\xi \in s(\tau,k)}   \phi_\xi = 1$ and linearity, we obtain
\begin{equation*}
\sum_{\xi\in s(\tau,k)} \ell_\tau(  \phi_\xi) = \sum_{\xi\in s(\tau,k)} \int_\tau f    \phi_\xi \ \text{d} \boldsymbol{x} = \int_\tau f \sum_{\xi\in s(\tau,k)}   \phi_\xi \ \text{d} \boldsymbol{x} = \int_\tau f \ \text{d} \boldsymbol{x}.
\end{equation*}
Using the fact that  $\nabla \big( \sum_{\xi\in s(\tau,k)}   \phi_\xi \big) = \boldsymbol{0} $ and linearity, we obtain
\begin{equation*}
\sum_{\xi\in s(\tau,k)} a_\tau(u_h,   \phi_\xi) = \sum_{\xi\in s(\tau,k)} \int_\tau \kappa  \nabla u_h \cdot \nabla   \phi_\xi \ \text{d} \boldsymbol{x} = \int_\tau \kappa  \nabla u_h \cdot \nabla \big( \sum_{\xi\in s(\tau,k)}   \phi_\xi \big) \text{d} \boldsymbol{x} = 0.
\end{equation*}
Also, we notice that 
\begin{equation*}
\sum_{\xi\in s(\tau,k)} e_\tau(u_h,   I_\tau  \phi_\xi -  \phi_\xi ) = \sum_{\xi\in s(\tau,k)}  \int_{\partial \tau \cap \partial t_\xi } \{ \kappa  \nabla u_h \} \cdot \boldsymbol{n}  \ \text{d} l - \int_{\partial \tau} \{ \kappa  \nabla u_h \} \cdot \boldsymbol{n} \sum_{\xi\in s(\tau,k)}   \phi_\xi \ \text{d} l = 0.
\end{equation*}
Combining these equalities, compatibility condition $\int_{\partial \tau} - \kappa  \nabla \widetilde u_{\tau, h} \cdot \boldsymbol{n} \ \text{d} l = \int_\tau f \ \text{d} \boldsymbol{x}$ is verified. 
This completes the proof.
\end{proof}

\begin{remark}
The technique proposed here can be naturally generalized to rectangular elements. In the proof of Lemma \ref{lem:ebvp}, for $\xi = z_{10}$ in cubic CGFEM, \eqref{eq:cz10} is naturally reduced to a local conservation equation $- \int_{\partial C^\xi } \kappa  \nabla \widetilde u_{\tau, h} \cdot \boldsymbol{n}  \ \text{d} l = \int_{C^\xi} f \ \text{d} \boldsymbol{x}$. This can be proved by using \eqref{eq:localfem}, $\partial \tau \cap \partial t_\xi = \emptyset$, and $\phi_{z_{10}} = 0$ on $\partial \tau.$
\end{remark}

\begin{lem}  \label{lem:elemlc}
The piecewise boundary fluxes defined in \eqref{eq:bvpb} satisfy local conservation on each element, i.e.,
\begin{equation} \label{eq:elemlc}
- \int_{\partial\tau } \kappa  \nabla \widetilde u_{\tau, h} \cdot \boldsymbol{n}  \ \text{d} l = \int_\tau f \ \text{d} \boldsymbol{x},
\end{equation}
and
\begin{equation} \label{eq:fluxed}
- \int_{\partial\tau } \kappa  \nabla \widetilde u_{\tau, h} \cdot \boldsymbol{n}  \ \text{d} l = - \int_{\partial\tau } \kappa  \nabla u \cdot \boldsymbol{n}  \ \text{d} l.
\end{equation}
\end{lem}

\begin{proof}
Equation \eqref{eq:elemlc} is established in the proof of Lemma \ref{lem:ebvp}. 
Identity \eqref{eq:fluxed} is verified by using \eqref{pde} and Divergence theorem:
\begin{equation*}
\begin{aligned}
\int_{\partial \tau} - \kappa  \nabla \widetilde u_{\tau, h} \cdot \boldsymbol{n} \ \text{d} l = \int_\tau f \ \text{d} \boldsymbol{x}  = \int_\tau \nabla \cdot ( -\kappa  \nabla u)   \ \text{d} \boldsymbol{x} = \int_{\partial\tau } -\kappa  \nabla u \cdot \boldsymbol{n}  \ \text{d} l.
\end{aligned} 
\end{equation*}

\end{proof}

\begin{remark}
Lemma \ref{lem:elemlc} implies that the proposed way of imposing elemental boundary condition in \eqref{eq:bvpb} is a rather simple post-processing technique: we set flux at $\partial \tau \cap \partial t_\xi $ as $\ell_\tau (  \phi_\xi ) - a_\tau(u_h,  \phi_\xi ) - e_\tau( u_h, I_\tau  \phi_\xi -  \phi_\xi )$. It does not require solving any linear system but provides locally conservative fluxes at the element boundaries. 
In two-dimensional case, $\partial \tau \cap \partial t_\xi $ consists of two segments in the setting as shown in Figure \ref{fig:elem}. If we want to provide a flux approximation on each segment, we can for example set the flux for each segment $\Gamma_\xi$ as 
\begin{equation}
\frac{\ell_\tau (  \phi_\xi ) - a_\tau(u_h,  \phi_\xi ) + e_\tau( u_h, \phi_\xi )}{2} - \int_{\Gamma_\xi} \{ \kappa  \nabla u_h \} \cdot \boldsymbol{n}  \ \text{d} l.
\end{equation}
If $\partial \tau \cap \partial t_\xi $ consists of more segments, we can set the flux in the similar way. For instance, we assign weights to $\ell_\tau ( \phi_\xi ) - a_\tau(u_h,  \phi_\xi ) + e_\tau( u_h, \phi_\xi )$ according to the length ratio of the segment over $\partial \tau \cap \partial t_\xi $. A drawback of this technique, however, is that the post-processed fluxes in general is not continuous at the element boundaries, except only when $\ell_\tau (  \phi_\xi ) - a_\tau(u_h,  \phi_\xi )$ from two neighboring elements are equal. This reveals the merits of the main post-processing technique proposed in this paper. The main post-processing technique provides a way to obtain locally conservative fluxes on control volumes and the fluxes are continuous at the boundaries of each control volume.

\end{remark}

\begin{lem}  \label{lem:ppsol}
The true solution $u$ of \eqref{pde} satisfies 
\begin{equation} \label{eq:bvpvft}
b_\tau(u, w) = \ell_\tau ( I_\tau w - w ) + a_\tau(u, w) + e_\tau(u, I_\tau w - w), \quad \forall \ w \in H^1(\tau),
\end{equation}
and this further implies
\begin{equation} \label{eq:bvpvfe}
b_\tau(u - \widetilde u_{\tau, h}, w) = a_\tau(u - u_h, w) + e_\tau(u - u_h, I_\tau w - w), \quad \forall \ w \in V^k_h(\tau).
\end{equation}
\end{lem}

\begin{proof}
This can be easily proved by simple calculations.
\end{proof}

\begin{remark} 
 The first result in Lemma \ref{lem:ppsol} tells us that accurate boundary data gives us a chance to obtain accurate fluxes. This is the very reason that the post-processing technique proposed in \cite{bush2013application} and extended in \cite{bush2014application, bush2015locally, deng2015construction} could not be generalized to a post-processing technique for high order CGFEMs. The boundary condition imposed in \cite{bush2013application} is
\begin{equation}
 \int_{\partial \tau \cap \partial t_\xi } - \kappa  \nabla \widetilde u_{\tau, h}  \cdot \boldsymbol{n}  \ \text{d} l = \ell_\tau(  \phi_\xi) - a_\tau(u_h,   \phi_\xi),
\end{equation} 
and clearly $\ell_\tau(w) - a_\tau(u, w) \ne 0$ for
the true solution $u$ of \eqref{pde}. Specifically for high order CGFEM solutions, only
imposing $\ell_\tau(w) - a_\tau(u_h, w)$ as a boundary condition for $\widetilde{u}_{\tau,h}$ is not sufficient to guarantee
optimal accuracy. The convergence order and accuracy of the post-processed solution strongly depend on the boundary data. By imposing the boundary data in the way shown in \eqref{eq:bvpb}, we can get optimal convergence order for the post-processed solution $\widetilde u_h$ for any high order CGFEMs.
\end{remark}

\begin{remark}
Equation \eqref{eq:bvpvfe} in Lemma \ref{lem:ppsol} resembles Galerkin orthogonality and plays a crucial role in establishing the post-processing error. 
 \end{remark}

\subsection{Elemental Linear System } \label{sec:llas}
We note that the dimension of $V^k_h(\tau)$ is $N_k$ and hence the variational formulation \eqref{eq:bvpvf} yields an $N_k$-by-$N_k$ linear algebra system. Since $\widetilde u_{\tau, h} \in V^k_h(\tau)$,
\begin{equation} \label{eq:ppsol}
\widetilde u_{\tau, h} = \sum_{\eta \in s(\tau,k)} \alpha_\eta   \phi_\eta,
\end{equation}
so by inserting this representation to \eqref{eq:bvpvf} and replacing the test function
by $\phi_\xi$ give us the linear  algebra system
\begin{equation} \label{eq:axb}
A \boldsymbol{\alpha} = \boldsymbol{\beta},
\end{equation}
where 
$\boldsymbol{\alpha} \in \mathbb{R}^{N_k} $ whose entries are the nodal solutions in \eqref{eq:ppsol}, $\boldsymbol{\beta} \in \mathbb{R}^{N_k} $ with entries
\begin{equation}
\beta_\xi = \ell_\tau ( I_\tau \phi_\xi - \phi_\xi ) + a_\tau(u_h,   \phi_\xi) + e_\tau (u_h,   I_\tau \phi_\xi - \phi_\xi ), \quad \forall \ \xi \in s(\tau,k),
\end{equation}
and
\begin{equation}
A_{\xi \eta} = b_\tau( \phi_\eta,   \phi_\xi), \quad \forall \ \xi, \eta \in s(\tau,k),
\end{equation}

The linear system \eqref{eq:axb} is singular and there are infinitely many solutions since the solution to \eqref{eq:bvpvf} is unique up to a constant by Lemma \ref{lem:ebvp}. However, this does not cause any issue since
 to obtain locally conservative fluxes, the desired quantity from the post-processing
is $\nabla \widetilde{u}_{\tau,h}$, which is unique.

\subsection{Local Conservation} \label{sec:loccons}
 
At this stage, we verify the local conservation property \eqref{eq:cvconservation} on control volumes for the post-processed solution. It is stated in the following lemma.

\begin{lem}  \label{lem:localconserv}
The desired local conservation property \eqref{eq:cvconservation} is satisfied on the control volume $C^\xi$ where $\xi \in Z_{\text{in}}$.
\end{lem}

\begin{proof}
Obviously, for $\xi = z_{10}$ in the case of $V_h^3$, the polygonal $t_{z_{10}} = C^{10}$, \eqref{eq:cvconservation}  is directly satisfied from solving \eqref{eq:bvpvf}. Similar situation occurs for $k>3$ in $V_h^k$. Thus, 
we only need to prove this lemma in the case that a control volume is associated with $\xi$ that is either on the edge of
$\tau$ or the vertex of $\tau$. For a basis function $\phi_\xi$, let $\Omega^\xi = \cup_{i=1}^{N_\xi} \tau_i^\xi$ be its support. Noting that the gradient component is averaged, it is obvious that 
\begin{equation*}
\sum_{j=1}^{N_\xi}  \int_{\partial \tau_j^\xi} \{ \kappa  \nabla u_{\tau_j, h} \} \cdot \boldsymbol{n} \phi_\xi  \ \text{d} l = 0 \qquad \text{and} \qquad \sum_{j=1}^{N_\xi} \int_{\partial \tau_j^\xi \cap \partial t_\xi } \{ \kappa  \nabla u_{\tau_j, h} \} \cdot \boldsymbol{n}  \ \text{d} l = 0.
\end{equation*}
This implies that $\sum_{j=1}^{N_\xi} e_{\tau_j^\xi}(u_h, \phi_\xi) = 0.$ 
Straightforward calculation and \eqref{eq:localfem} gives
\begin{equation*}
 \int_{ \partial C^\xi } - \kappa  \nabla \widetilde u_{\tau, h} \cdot \boldsymbol{n} \ \text{d} l =  \int_{C^\xi} f \ \text{d} \boldsymbol{x} + \sum_{j=1}^{N_\xi} \Big( a_{\tau_j^\xi}(u_h, \phi_\xi) - \ell_{\tau_j^\xi}(\phi_\xi) - e_{\tau_j^\xi}(u_h, \phi_\xi) \Big) = \int_{C^\xi} f \ \text{d} \boldsymbol{x},
\end{equation*}
which completes the proof.
\end{proof}

\section{An Error Analysis for the Post-processing}  \label{sec:ana}

In this section, we focus on establishing an optimal convergence property of the post-processed solution $\widetilde u_{\tau, h}$ in $H^1$ semi-norm. We denote $\| \cdot \|_{L^2}$ the usual $L^2$ norm and $| \cdot |_{W}$ the usual semi-norm in a Sobolev space $W$. We start with proving a property of $I_\tau$ defined in Section \ref{sec:bvp}.

\begin{lem}  \label{lem:map}
Let $I_\tau$ be as defined in Section \ref{sec:bvp}. Then 
\begin{equation} \label{eq:interr}
\| w - I_\tau w \|_{L^2(\tau)} \leq Ch^2_\tau | w |_{H^2(\tau)} + C h_\tau | w |_{H^1(\tau)}, \quad \text{for} \quad w \in H^2(\tau).
\end{equation}
\end{lem}

\begin{proof}
For $w \in H^2(\tau)$, suppose $\Pi w \in V_h^1(\tau)$ is the standard linear interpolation of $w$. Then we have $I_\tau w = I_\tau (\Pi w)$. By adding and subtracting $\Pi w$, and invoking triangle inequality we get
\begin{equation*} 
\begin{aligned}
\| w - I_\tau w \|_{L^2(\tau)} \leq \| w - \Pi w \|_{L^2(\tau)} + \|  \Pi w - I_\tau (\Pi w) \|_{L^2(\tau)}.
\end{aligned}
\end{equation*}
Standard interpolation theory (see for example Theorem 4.2 in \cite{johnson2009numerical}) states that
\begin{equation}
\| w - \Pi w \|_{L^2(\tau)} \leq Ch^2_\tau | w |_{H^2(\tau)}.
\end{equation}
Since $I_\tau w = \displaystyle \sum_{\xi \in s(\tau,k)} w_\xi \psi_\xi$, we divide $\tau$ equally into $k^2$ sub-triangles $\tau_j, j=1, \cdots, k^2$. By Lemma 6.1 in \cite{chatzipantelidis2002finite}, we have
\begin{equation}
\|  \Pi w - I_\tau (\Pi w) \|^2_{L^2(\tau)} = \sum_{j=1}^{k^2} \|  \Pi w - I_\tau (\Pi w) \|^2_{L^2(\tau_j)} \leq \sum_{j=1}^{k^2} C h^2_{\tau_j} | \Pi w |^2_{H^1(\tau_j)} \leq Ch^2_\tau | \Pi w |^2_{H^1(\tau)}.
\end{equation}
Taking square root gives
\begin{equation}
\|  \Pi w - I_\tau (\Pi w) \|_{L^2(\tau)} \leq Ch_\tau | \Pi w |_{H^1(\tau)}.
\end{equation}
 By using triangle inequality and interpolation theory again, we have 
$$
|  \Pi w |_{H^1(\tau)} \leq | w |_{H^1(\tau)} + | w - \Pi w |_{H^1(\tau)} \leq | w |_{H^1(\tau)} + Ch_\tau | w |_{H^2(\tau)}.
$$
Combining these inequalities gives the desired result.
\end{proof}

\begin{lem}  \label{lem:loccoe}
The bilinear form defined in \eqref{eq:bvpvf} is bounded, i.e., 
for all $w \in H^2(\tau), v \in V^k_h(\tau),$
\begin{equation} \label{eq:bbounded}
b_\tau(w, v) \leq C |w|_{H^1(\tau)} |v|_{H^1(\tau)}.
\end{equation}
Furthermore, for $v\in V^k_h(\tau)$ with $k=1, 2$, $b_\tau(\cdot,\cdot)$ is coercive, namely,  
\begin{equation} \label{eq:coerc}
b_\tau(v, v) \geq C_b |v|^2_{H^1(\tau)},
\end{equation}
for some positive constant $C_b$. 

\end{lem}

\begin{proof}
The boundedness of $b_\tau(\cdot, \cdot)$ has been established in Theorem 1 in \cite{xu2009analysis}. The local coercivity is also established for linear (Theorem 2) and quadratic (Theorem 5) CGFEM in \cite{xu2009analysis}.
\end{proof}

\begin{lem}  \label{lem:eb}
Fix a triangle $\tau = \tau_0.$ Suppose $\{ \tau_i \}_{i=1}^3$ are the neighbors (sharing edges) of $\tau$, i.e., $\partial \tau \cap \partial \tau_i \neq \emptyset.$ Then for $w, v\in H^2(\tau)$
\begin{equation} \label{eq:ebounded}
e_\tau(w, I_\tau v - v) \leq C  h_\tau^{1/2} \Big(  |v|_{H^1(\tau)}  + h_\tau |v|_{H^2(\tau)}  \Big) \sum_{i=0}^3  \Big( h_{\tau_i}^{-1/2} | w  |_{H^1(\tau_i)} + h_{\tau_i}^{1/2} | w  |_{H^2(\tau_i)}  \Big),
\end{equation}
where $C$ is a constant independent on $h_\tau$ and $h_{\tau_i}.$
\end{lem}

\begin{proof}
By definition and Cauchy-Schwarz inequality, 
\begin{equation} \label{eq:ede0}
\begin{aligned}
e_\tau(w, I_\tau v - v) & = \int_{\partial \tau} \{ \kappa  \nabla w \} \cdot \boldsymbol{n} ( I_\tau v -  v )  \ \text{d} l \\
& \leq \Big( \int_{\partial \tau} | \{ \kappa  \nabla w \} \cdot \boldsymbol{n} |^2 \ \text{d} l \Big)^{1/2}  \Big( \int_{\partial \tau} | I_\tau v -  v  |^2  \ \text{d} l \Big)^{1/2}  \\
& \leq \kappa_{\tau, \text{max}} \| \{ \nabla w \} \cdot \boldsymbol{n}  \|_{L^2(\partial \tau)}  \| I_\tau v -  v  \|_{L^2(\partial \tau)},
\end{aligned}
\end{equation}
where $\kappa_{\tau, \text{max}}$ is the maximum of $\kappa$ on $\tau.$
 By trace inequality, we have
\begin{equation} \label{eq:ede1}
\| \{ \nabla w \} \cdot \boldsymbol{n}  \|_{L^2(\partial \tau)}  \leq \frac{1}{2} \sum_{i=0}^3  \| \nabla w \cdot \boldsymbol{n}  \|_{L^2(\partial \tau_i)}   \leq \frac{1}{2}  \sum_{i=0}^3   \Big( C h_{\tau_i}^{-1/2} | w  |_{H^1(\tau_i)} + C h_{\tau_i}^{1/2} | w  |_{H^2(\tau_i)}  \Big).
\end{equation}
Similarly by trace inequality and Lemma \ref{lem:map}, we have
\begin{equation} \label{eq:ede2}
\begin{aligned}
\| I_\tau v -  v  \|_{L^2(\partial \tau)}  & \leq \sum_{\xi \in s(\tau,k)} \| I_\tau v -  v  \|_{L^2(\partial t_\xi)}  \\
& \leq \sum_{\xi \in s(\tau,k)} C h_\tau^{-1/2} ||  I_\tau v -  v ||_{L^2(t_\xi)} + C h_\tau^{1/2} | v  |_{H^1(t_\xi)} \\
& \leq C \Big( h_\tau^{-1/2} ||  I_\tau v -  v ||_{L^2(\tau)} + h_\tau^{1/2} | v  |_{H^1(\tau)} \Big) \\
& \leq C h_\tau^{-1/2} \Big( C h^2_\tau |v|_{H^2(\tau)} + C h_\tau |v|_{H^1(\tau)}  \Big)+ C h_\tau^{1/2} | v  |_{H^1(\tau)} \\
& \leq C h_\tau^{1/2} \Big(  |v|_{H^1(\tau)}  + h_\tau |v|_{H^2(\tau)} \Big),
\end{aligned}
\end{equation}
where $t_\xi$ are the polygonals defined in Figure \ref{fig:elem}.
Putting \eqref{eq:ede1} and \eqref{eq:ede2} into \eqref{eq:ede0} gives us the desired result.

\end{proof}

\begin{lem}  \label{lem:locerr}
We have the following local error estimate
\begin{equation}
|u - \widetilde u_{\tau, h} |_{H^1{(\tau)}} \leq C | u - u_h |_{H^1{(\tau)}} + Ch_\tau^{1/2}  \sum_{i=0}^3   \Big( h_{\tau_i}^{-1/2} | u - u_h  |_{H^1(\tau_i)} + h_{\tau_i}^{1/2}  | u - u_h  |_{H^2(\tau_i)}  \Big).
\end{equation}
\end{lem}

\begin{proof}
Triangle inequality gives
\begin{equation} \label{eq:j1}
| u - \widetilde u_{\tau, h} |_{H^1{(\tau)}}  \leq | u - u_h |_{H^1{(\tau)}} + | u_h - \widetilde u_{\tau, h} |_{H^1{(\tau)}}.
\end{equation}

By Lemma \ref{lem:loccoe}, we have
\begin{equation}  \label{eq:j2}
\begin{aligned}
C_b | u_h - \widetilde u_{\tau, h} |^2_{H^1{(\tau)}}&  \leq b_\tau(u_h - \widetilde u_{\tau, h}, u_h - \widetilde u_{\tau, h} ) \\
& = b_\tau(u_h - u, u_h - \widetilde u_{\tau, h} ) + b_\tau(u - \widetilde u_{\tau, h}, u_h - \widetilde u_{\tau, h} ) \\
& \leq C | u - u_h |_{H^1{(\tau)}} | u_h - \widetilde u_{\tau, h}  |_{H^1{(\tau)}} + b_\tau(u - \widetilde u_{\tau, h}, u_h - \widetilde u_{\tau, h} ). \\
\end{aligned}
\end{equation}

For simplicity, we set $\delta_{\tau, h} = u_h - \widetilde u_{\tau, h}$.  By \eqref{eq:bvpvfe}, we have
\begin{equation}\label{eq:j3}
b_\tau(u - \widetilde u_{\tau, h}, \delta_{\tau, h} ) = a_\tau(u - u_h, \delta_{\tau, h} ) + e_\tau(u - u_h, \delta_{\tau, h} )
\end{equation}

Now, by boundedness of the bilinear form of $a_\tau(\cdot, \cdot)$,
\begin{equation}\label{eq:j4}
a_\tau(u - u_h, \delta_{\tau, h}) \leq \kappa_{\tau, \text{max}} | u - u_h |_{H^1{(\tau)}} | \delta_{\tau, h} |_{H^1{(\tau)}},
\end{equation}
where $\kappa_{\tau, \text{max}}$ is the maximum of $\kappa$ on $\tau.$
By Lemma \ref{lem:eb} and inverse inequality, we obtain
\begin{equation}  \label{eq:j5}
\begin{aligned}
 e_\tau(u - u_h, \delta_{\tau, h} ) & \leq C  h_\tau^{1/2} \Big(  |\delta_{\tau, h}|_{H^1(\tau)}  + h_\tau |\delta_{\tau, h}|_{H^2(\tau)} \Big) \sum_{i=0}^3  \Big( h_{\tau_i}^{-1/2} | u - u_h  |_{H^1(\tau_i)} + h_{\tau_i}^{1/2} | u - u_h |_{H^2(\tau_i)}  \Big) \\
 &  \leq  Ch_\tau^{1/2}  | \delta_{\tau, h} |_{H^1(\tau)} \sum_{i=0}^3   \Big( h_{\tau_i}^{-1/2} | u - u_h  |_{H^1(\tau_i)} + h_{\tau_i}^{1/2} | u - u_h  |_{H^2(\tau_i)}  \Big). 
 \end{aligned}
\end{equation}

Combining \eqref{eq:j2}, \eqref{eq:j3}, \eqref{eq:j4}, and \eqref{eq:j5}, dividing both side $ | u_h - \widetilde u_{\tau, h} |_{H^1{(\tau)}}$ and then putting it into \eqref{eq:j1} gives the desired result. 
\end{proof}

\newtheorem{thm}{Theorem}[section]
\begin{thm} \label{thm:perr}
Assume $u$ is the solution of \eqref{pde} and it is sufficiently smooth. Let $\widetilde u_h = \sum_{\tau \in \mathcal{T}_h } \widetilde u_{\tau, h} \chi_\tau$, where $\widetilde u_{\tau, h} \in V^k_h(\tau)$ is the  post-proccessed solution \eqref{eq:ppsol} and $\chi_\tau$ is the usual characteristic function for $\tau$, then we have
$$
| u - \widetilde u_h |_{H^1(\Omega)} \leq Ch^k |u|_{H^{k+1}(\Omega)},
$$
where $C$ is a constant independent of $h$.
\end{thm}

\begin{proof}
Noticing that $h_\tau \leq h$, this can be proved by using Lemma \ref{lem:locerr} and arithmetic-geometric mean inequality:
\begin{equation*}
\begin{aligned}
| u - \widetilde u_h |^2_{H^1(\Omega)} & = \sum_{\tau} | u - \widetilde u_h |^2_{H^1(\tau)}  \\
& \leq
C \sum_{\tau} \Big( | u - u_h |^2_{H^1{(\tau)}} + h^2 | u - u_h |^2_{H^2{(\tau)}} \Big) \\
& \leq C | u - u_h |^2_{H^1(\Omega)} + C h^2 | u - u_h |^2_{H^2(\Omega)} .
\end{aligned}
\end{equation*}
By the property of CGFEM approximation (see Theorem 14.3.3 in \cite{brenner2008mathematical} for example), we have
\begin{equation}
| u - u_h |_{H^1(\Omega)} \leq C h^k |u|_{H^{k+1}(\Omega)}, \qquad | u - u_h |_{H^2(\Omega)} \leq C h^{k-1} |u|_{H^{k+1}(\Omega)}.
\end{equation}
Substituting these inequalities back gives the desired result.

\end{proof}

\section{Numerical Experiments}\label{sec:num}

In this section we present various numerical examples to illustrate the performance of the proposed post-processing technique for  CGFEM using $V_h^k$, $k=1,2,3$. For the numerical examples, we consider mainly the local conservation property of the post-processed fluxes and the convergence behavior of the post-processed solutions. For these purposes, we consider the following test problems in the unit domain $[0, 1]\times[0, 1]$.

\textbf{Example 1.} Elliptic equation \eqref{pde} with $\kappa =1, u = (x-x^2) (y-y^2)$ with fully Dirichlet boundary condition $g = 0$. $f$ is the function derived from \eqref{pde}.

\textbf{Example 2.} Elliptic equation \eqref{pde} with $\kappa =e^{2x-y^2}, f = -e^{x}, u = e^{-x + y^2}$ with the fully Dirichlet boundary condition $g$ satisfying the true solution.

\textbf{Example 3.} Elliptic equation \eqref{pde} with 
$$
\kappa = \frac{1}{1-0.8\sin(6\pi x)} \cdot \frac{1}{1-0.8\sin(6\pi y)},
$$
$$
u = 1 - \frac{2\cos(6\pi x) + 15\pi x - 2}{15 \pi},
$$
and $f = 0$. We impose boundary conditions as Dirichlet $1$ at the left boundary and $0$ at the right boundary with
homogeneous Neumann boundary conditions at the top and bottom boundaries. 

For each test problem, we will present the numerical results of linear, quadratic, and cubic CGFEMs. Now we start with a study of local conservation property of the post-processed fluxes.

\subsection{Conservation Study} \label{sec:cns}
To numerically illustrate the behavior of the post-processed fluxes, we run the examples by verifying that the post-processed fluxes satisfies the desired local conservation property \eqref{eq:cvconservation}. For this purpose, we define a local conservation error (LCE) as
\begin{equation}
\text{LCE}_z = \int_{\partial C^z}  - \kappa  \nabla \hat u_h \cdot \boldsymbol{n}  \ \text{d} l - \int_{C^z} f \ \text{d} \boldsymbol{x},
\end{equation}
where $\hat u_h = u_h$ for CGFEMs solution and $\hat u_h = \widetilde u_h$ for the post-processed solution. Naturally, $\text{LCE}_z=0$ means local conservation property \eqref{eq:cvconservation} is satisfied while $\text{LCE}_z \ne 0$ means local conservation property \eqref{eq:cvconservation} is not satisfied on the control volume $C^z$. 

Without a post-processing, for instance, LCEs of $u_h$  solved by quadratic CGFEMs are shown by red plots in the left column for Example 1 and right column for Example 2 in Figure \ref{fig:ex12lce}, respectively. We see that these errors are non-zeros, which means that the local conservation property \eqref{eq:cvconservation} is not satisfied.  The control volume indices in the figures are arranged as follows: firstly indices from vertices of the mesh, secondly the indices of the degrees of freedom on edges of elements, and lastly (for the cubic case) indices of degrees of freedom inside the elements. Now with the post-processing, LCEs of $\widetilde u_h$ in the quadratic case are shown by dotted green plots in the left column for Example 1 and right column for Example 2 in Figure \ref{fig:ex12lce}, respectively. These errors are practically negligible, which is mainly attributed to the errors in the application of numerical integration and the machine precision. Theoretically, these errors should be zeros as discussed in Section \ref{sec:loccons}. The LCEs for both $u_h$ and $\widetilde u_h$ for Example 3  are shown in Figure \ref{fig:ex3lce}. 

\begin{figure}[ht]
\centering
\includegraphics[height=3.5cm]{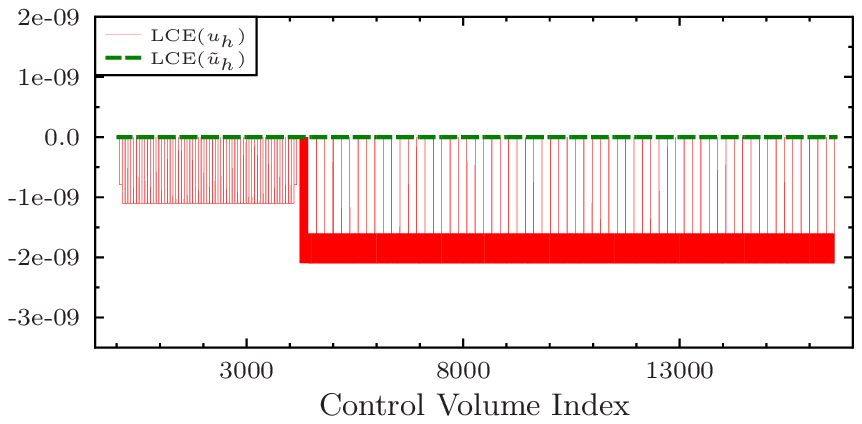}
\includegraphics[height=3.5cm]{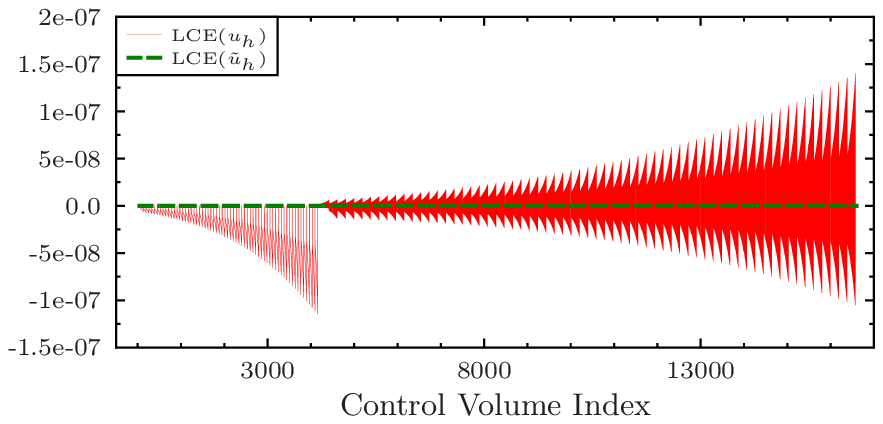}
\caption{ LCEs for $u_h$ (red plots) and for $\widetilde u_h$ (dotted green plots) for Example 1 (left column) and Example 2 (right column) using CGFEM with $V_h^2$.}
\label{fig:ex12lce}
\end{figure}

\begin{figure}[ht]
\centering
\includegraphics[height=4.5cm]{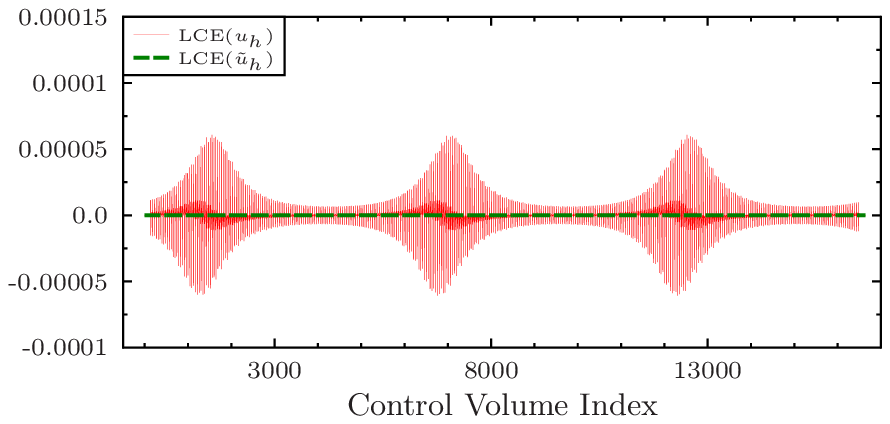} \includegraphics[height=4.5cm]{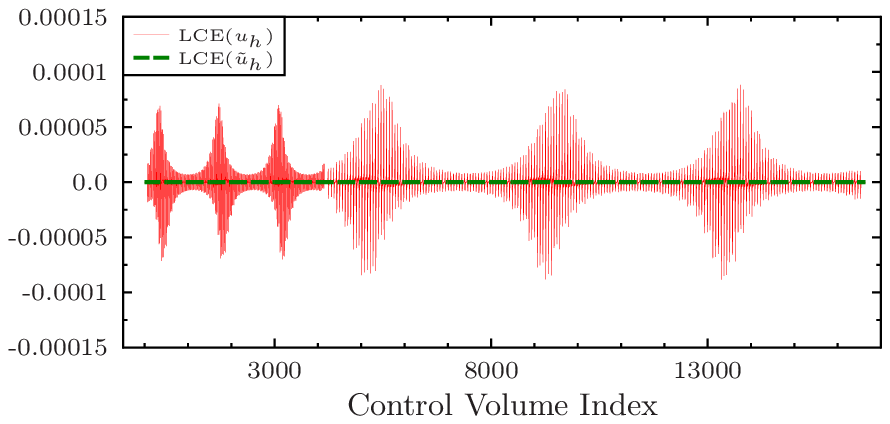}
\includegraphics[height=4.5cm]{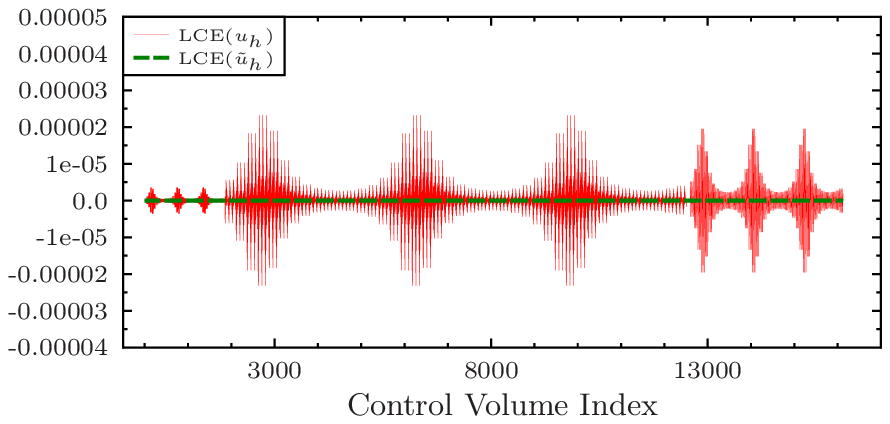}
\caption{ LCEs for $u_h$ (red plots) and for $\widetilde u_h$ (dotted green plots) for Example 3 using $V_h^k$ with 
$k=1$ (top), $k=2$ (middle), and $k=3$ (bottom). }
\label{fig:ex3lce}
\end{figure}

\subsection{Convergence Study} \label{sec:cnv}
Now we show the numerical convergence rates for Example 1, 2, and 3. We collect the $H^1$ semi-norm errors of the CGFEM solution, post-processed solution, and also the difference between these two solutions defined as $| u_h - \widetilde u_h |_{H^1(\Omega)}$. The results for Example 1, 2, and 3 are shown in Figure \ref{fig:ex123h1err}. The $H^1$ semi-norm errors of $k^{\text{th}}$ order CGFEM solutions and post-processed solutions are of optimal convergence orders, which confirm convergence analysis in Theorem \ref{thm:perr} in Section \ref{sec:ana}.

From Figure \ref{fig:ex123h1err},  we can also see that for quadratic and cubic CGFEM, $| u_h - \widetilde u_h |_{H^1(\Omega)}$ tends to be good error estimators. In the linear case, $| u_h - \widetilde u_h |_{H^1(\Omega)}$ is of order 2, which is higher than the optimal error convergence order of CGFEM solution.

\begin{figure}[ht]
\centering
\includegraphics[height=5cm]{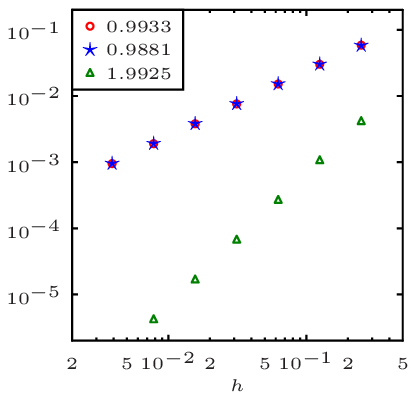}
\includegraphics[height=5cm]{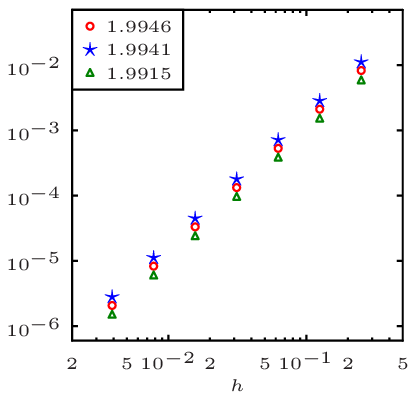}
\includegraphics[height=5cm]{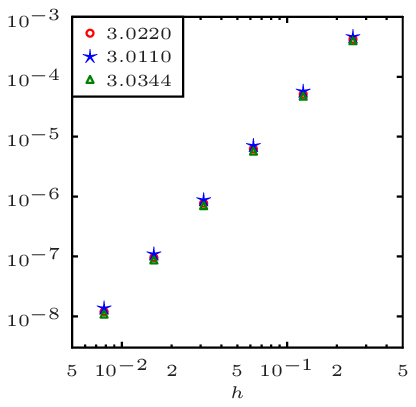}

\includegraphics[height=5cm]{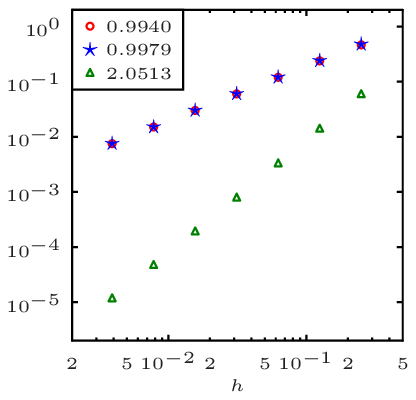}
\includegraphics[height=5cm]{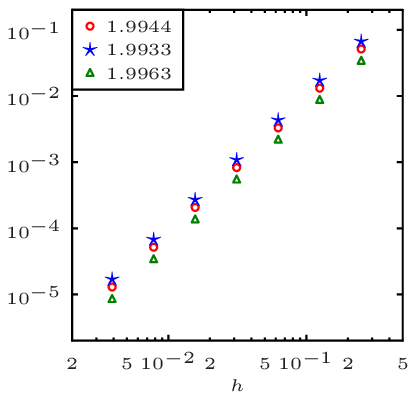}
\includegraphics[height=5cm]{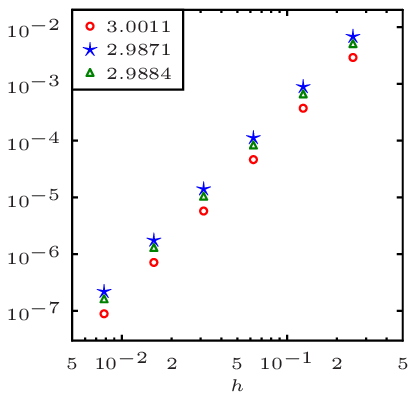}

\includegraphics[height=5cm]{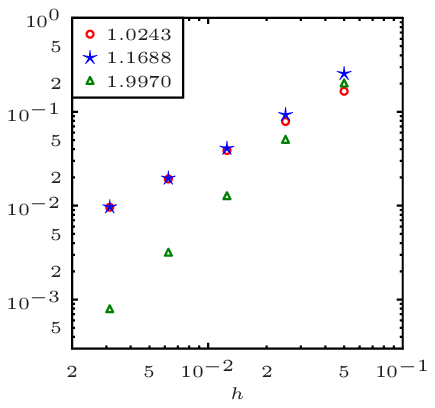}
\includegraphics[height=5cm]{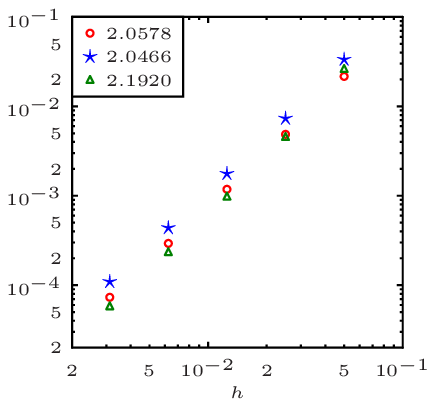}
\includegraphics[height=5cm]{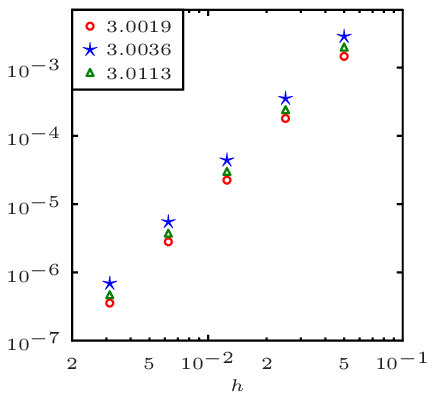}
\caption{ $H^1$ semi-norm errors: \textcolor{red}{$\circ$} is for $|u - u_h|_{H^1(\Omega)}$,
\textcolor{blue}{$\star$} is for  $|u - \widetilde u_h|_{H^1(\Omega)}$, {\smaller \smaller \smaller \textcolor{green}{$\triangle$}} is for $|u_h - \widetilde u_h|_{H^1(\Omega)}$ for Example 1 (top row), 2 (middle row), and 3 (bottom row) using linear (left column), quadratic (middle column), and cubic (right column) CGFEM. }
\label{fig:ex123h1err}
\end{figure}

\section{Conclusion}
In this work, we proposed a post-processing technique for any order CGFEM for elliptic problems. This technique builds a bridge between any order of finite volume element method (FVEM) and any order of CGFEM. FVEM has the advantage of its local conservation property but the disadvantages in analysis especially for higher order FVEMs, while CGFEM has its fully established analysis but it lacks the local conservation property. This technique is naturally proposed in this less than ideal  situation to serve as a great tool if one would like to use CGFEM to solve PDEs while maintaining a local conservation property, for instance in two-phase flow simulations. 

Since the problems that the post-processing technique requires to solve are localized, they are independent of each other. For linear CGFEM, the technique requires solving a 3-by-3 system for each element while for quadratic and cubic CGFEMs, the technique requires solving a 6-by-6 system and 10-by-10 system for each element, respectively.  It thus can be easily, naturally, and efficiently implemented in a parallel computing environment.

As for future work, one can use this technique for other differential equations, such as advection diffusion equations, two-phase flow problems, and elasticity models; also one can apply this technique for other numerical methods, such as SUPG. One interesting direction we would like to work on is to develop a post-processing technique which requires solving only 3-by-3 systems for each element and for any order of CGFEMs. 


\bibliographystyle{siam}
\bibliography{pphofem}

\end{document}